\documentclass[11pt]{amsart}

\usepackage{amsmath}
\usepackage{amssymb}
\usepackage{graphicx}
\usepackage{color}

\newtheorem{theorem}{Theorem}[section]

\newtheorem{definition}[theorem]{Definition}

\newtheorem{corollary}[theorem]{Corollary}
\newtheorem{lemma}[theorem]{Lemma}
\theoremstyle{definition}
\newtheorem{example}[theorem]{Example}
\newtheorem{remark}[theorem]{Remark}
\newtheorem{questions}[theorem]{Questions}

\def\a{\alpha}
\def\be{\beta}
\def\z{\zeta}
\def\d{\delta}
\def\D{\Delta}
\def\e{\epsilon}
\def\E{\varepsilon}
\def\la{\lambda}
\def\g{\gamma}
\def\t{\tau}
\def\O{\Omega}
\def\o{\omega}
\def\s{\sigma}
\def\S{\Sigma}
\def\th{\theta}
\def\x{\{x_n\}}
\def\f{\frac{\| x_n-y\|}{1-\| x\|}}
\def\fr{\frac{\psi (t)}{\psi (\d t)}}
\def\|{\Vert}

\title[ MAPPINGS OF ASYMPTOTICALLY NONEXPANSIVE TYPE]
{ FIXED POINTS FOR MAPPINGS OF ASYMPTOTICALLY NONEXPANSIVE TYPE }
\author[T. Dom\'inguez Benavides; P. Lorenzo]{}

\author[T. Dom\'{\i}nguez ]{T. Dom\'{\i}nguez Benavides}
\address[T. Dom\'{\i}nguez Benavides]{Departamento de  An\'{a}lisis Matem\'{a}tico, Universidad de
Sevilla, Tarfia s/n, 41012-Sevilla, Spain}
\email{\tt tomasd@us.es}

\author[P. Lorenzo ]{P. Lorenzo Ram\'{\i}rez}
 \address[P. Lorenzo Ram\'{\i}rez] {Departamento de  An\'{a}lisis Matem\'{a}tico, Universidad de
Sevilla, Tarfia s/n, 41012-Sevilla, Spain}
 \email{\tt ploren@us.es}

\thanks{The   authors are  supported by MICIU,
Grant PGC2018-098474-B-C2-1 and Andalusian Regional Government Grant
FQM-127.}

\begin{document}

\begin{abstract}
We prove the existence of fixed points for mappings which satisfy some asymptotic  nonexpansive conditions  in Banach spaces which are either nearly uniformly convex or they satisfy that asymptotic centers of bounded sequence are compact. Nominally, we consider pointwise eventually nonexpansive mappings, pointwise asymptotically nonexpansive mappings and asymptotically type nonexpansive mappings. We do not assume the existence of a continuous iterated, solving some long-standing open questions about existence of a fixed point for these mappings in absence of continuity \cite{KX}.

\end{abstract}

\maketitle

\setcounter{page}{1}

\def\N{{\Bbb N}}
\def\R{{\Bbb R}}
\def\acc#1{\if i#1\accent"13 \char "10 %
    \else \if j#1\accent"13 \char"11 %
                \else \accent"13 #1\fi\fi }%

\font\tit=cmr12 at 14pt
\def\s{\vskip 0.3cm}
\font\texto=cmr12.tfm
\font\peq=cmr12 at 9pt
%\font\texto=cmr12 at 14pt
%\def\l{\lambda}

%\def\e{\epsilon}
\def\b{\bigskip}
\def\m{\medskip}

\def\a{\alpha}
\def\be{\beta}
\def\z{\zeta}
\def\d{\delta}
\def\D{\Delta}
\def\e{\epsilon}
\def\E{\varepsilon}
\def\la{\lambda}
\def\g{\gamma}
\def\t{\tau}
\def\O{\Omega}
\def\o{\omega}
\def\s{\sigma}
\def\S{\Sigma}
\def\th{\theta}
\def\x{\{x_n\}}
\def\dis{\displaystyle}
\def\l{\liminf_{n\to \infty}\frac{\| x_n-y\|}{1-\| x\|}}
\def\lb{\liminf_n\| x_n-x\| \leq \be}
\def\f{\frac{\| x_n-y\|}{1-\| x\|}}
\def\fr{\frac{\psi (t)}{\psi (\d t)}}
\def\V{\Vert}
\def\ma{\mathit{|||}}

%11111111111111111111111111111111111111111111111111111111111111111111111111111111111111
\vskip1cm

%\section{Introduction}

\section{Nonexpansive and eventually nonexpansive mappings}

Let $(X,\Vert\cdot \Vert)$ be a Banach space and $C$ a nonempty subset of $X$. A mapping $T:C\to C$ is said to be \emph{nonexpansive} if for each $x,y\in C$, $\Vert Tx-Ty\Vert\leq \Vert x-y\Vert$. The Banach space $X$ satisfies \emph{the fixed point property for nonexpansive mappings} (in short, FPP) if every nonexpansive mapping defined on a nonempty weakly compact convex subset $C$ of $X$ into $C$ has a fixed point. One of the central goals in metric fixed point theory  is to characterize those Banach spaces  which have the  FPP.  In the last 60 years a large number of papers have appeared finding out some geometrical properties that imply the FPP (for instance, see the monographs \cite{GK2, KS}  and the references therein).

%For instance it is known that all classic reflexive Banach spaces (finite-dimensional spaces, Hilbert spaces, uniformly convex spaces as $Lp(\Omega)$, $1< p< \infty)$ satisfy the FPP.

A natural relax in the nonexpansive assumption is to assume that the mapping $T$ is eventually nonexpansive, i.e.

\begin{definition}
	Let $X$ be a Banach  space and $C$ a nonempty subset of $X$. A mapping $T:C\to C$ is said to be eventually nonexpansive if there exists $N\in \N$ such that for every $n\geq N$
	$$
	\Vert T^nx-T^ny\Vert\leq \Vert x-y\Vert, \text{ for every }x, y \in X.
	$$
\end{definition}

It should be noted that an eventually nonexpansive mapping does not need to be nonexpansive, nor even continuous.
\begin{example}\label{e1}
	Let $C=[0,1]$ and $T:[0,1]\to [0,1]$ defined by $T(x)=0$ if $x< 1$ and $T(1)=1/2$. It is clear that $T$ is discontinuous at $x=1$ but $T^n\equiv 0$ for every  $n\geq 2$.
\end{example}

Looking at this example, we could guess that the fixed point theory for eventually nonexpansive mappings should be quite different of the corresponding theory for nonexpansive mapping. However, noting  that $T^n$ and $T^{n+1}$ are two commuting mappings, the equivalence between both theories is a direct consequence of  the following result:

\begin{theorem}[\cite{B2}]
	Let $X$ be a Banach space which satisfies the FPP, $C$ a weakly compact convex subset of $X$ and $\{T_i:i\in I\}$ an arbitrary family of commuting self-nonexpansive mappings of $C$. Then, the common fixed point set for this family is a nonempty nonexpansive retract of $C$.
\end{theorem}

Since $T^n(x)$ is  a fixed point of $T$ whenever $x$ is a common fixed point of $T^n$ and $T^{n+1}$, the following result easily follows.

\begin{theorem}[\cite{Ki2}] Let $X$ be a Banach space which satisfies the FPP. Then, every eventually nonexpansive mapping $T$ defined from a weakly compact convex set $C$ into $C$ has a fixed point.
\end{theorem}

We can also consider mappings which are eventually nonexpansive in a weaker sense, namely: for each $x,y\in C$ there exists $N=N(x,y)\in \N$ such that $\Vert T^n x-T^n y\Vert \leq \Vert x-y\Vert$ for $n\geq N$. The following easy   example shows the difficulty to obtain a fixed point in this case.
\begin{example}\label{e2}
	 Define $T:[0,1]\to [0,1]$ by $T (x)= x/2$ if $0< x\leq 1$ and $T(0)=1$. It is  clear that $T$ is fixed point free and for each $x\in [0,1]$ we have $T^n x\leq 2^{1-n}$. Thus, $T^n(x)\to 0$ as $n\to \infty$ for all $x\in [0,1]$ and $\displaystyle \lim_n\Vert T^n x-T^n y\Vert=0,  \textrm{ for all }x,y\in [0,1].$ \end{example}

In this paper we will consider three scaled notions of asymptotic non-expansivity which have been considered in the literature:   pointwise  eventually nonexpansive mappings, pointwise  asymptotically nonexpansive mappings and asymptotically nonexpansive type mappings. We will revise some known results about existence of fixed points for these classes  of mappings  and we will state some new results under more general assumptions. In particular, we  solve several long-standing open problems concerning the existence of fixed points for these mappings in  absence of continuity \cite{KX}.

In Section 2, we  recall some geometric properties that we will need in order to state our theorems, namely,  uniform normal structure, weak uniform normal structure, uniform convexity and nearly uniform convexity.

Section 3 is dedicated to introduce the different asymptotic  nonexpansive conditions which have been considered in the literature. We  revise the most relevant fixed point results in this setting \cite{Ki1, Ki3, KiX,  GS, Xu} where it is assumed that the mapping $T$ has a continuous iterated and we solve several  questions which appear as open problems in \cite{KX, Ki3}. We include a quite technical example which shows that a pointwise eventually nonexpansive mapping does not have, in general, a continuous iterated, even if the domain is a compact set.

In Section 4 we prove several  lemmas that will be needed to state our main theorems. Lemma  \ref{Xu} is taken from \cite{Xu}. This lemma, jointly with  Lemmas \ref{invariant} \label{measure}, yield to Lemma \ref{main} which is a basic tool in order to prove our fixed point results.

In Section 5 we present our main results. Nominally, we prove the existence of fixed a point for a pointwise  asymptotically nonexpansive mapping $T:C\to C$  when $C$ is a bounded convex closed subset of a Banach space $X$ and  one of the following conditions is satisfied:
\begin{itemize}

\item $C$ is a compact set.

\item $X$ is nearly uniformly convex.
\end{itemize}
and we also prove the existence of fixed a point for a pointwise eventually nonexpansive mapping when the asymptotic center of any sequence in $C$ is a compact set.

We include an example which shows that the result for nearly uniformly convex spaces cannot be derived from Theorem 3.17 in \cite{BT} as claimed by the authors.

 A final section includes several problems concerning the existence of fixed point for pointwise eventually nonexpansive mappings which still remain  open.

\section{Preliminaries}

Henceforth, $(X,\Vert\cdot \Vert)$ will  be a Banach space with unit ball $B(0,1)$ and $C$ a nonempty weakly compact convex subset of $X$.

\medskip

Let $\{x_n\}$ be a bounded sequence  in $X$, the \emph{asymptotic radius} of  $\{x_n\}$ with respect to a subset $C$ of $X$ is given by
$$
r(C,\{x_n\})=\inf\{\limsup_n\Vert x_n-x\Vert: x\in C\},
$$
and the \emph{asymptotic center} of $\{x_n\}$ is the set
$$
A(C,\{x_n\})=\{x\in C: \limsup_n\Vert x_n-x\Vert=r(C,\{x_n\})\}.
$$
It is known that  $A(C,\{x_n\})$ is a nonempty  weakly compact convex set as $C$ is. Whenever $C= \overline{\textrm{co}}\, (\{x_n\})$ we write $r(\{x_n\})$ instead of $r(C,\{x_n\})$.

\medskip

We recall some  properties which can be satisfied by $X$.
\begin{definition}
	It is said that $X$ has \emph{uniform normal structure} if $N(X)> 1$, where $N(X)$ is the normal structure coefficient of $X$ defined by
	$$
	N(X)=\inf\left\{\frac{\text{diam}(A)}{r(A)}: A\subset X\text{ bounded closed convex },\text{diam}(A)>0\right\},
	$$
	where $\text{diam}(A)=\sup\{\Vert x-y\Vert:x,y\in A\}$ is the diameter of $A$ and $r(A)=\inf\{\sup\{\Vert x-y\Vert:y\in A\}:x\in A\}$ is the Chebyshev radius of $A$.	

The space is said to have \emph{weak uniform normal structure} if $WCS(X)>1$ where
	$$ WCS(X)=\inf\left\{\frac{\text{diam}_a\,(\{x_n\})}{r(\{x_n\})}\right\},$$   the infimum runs over all weakly convergent sequences  $ \{x_n\}$ which are not norm convergent and $\text{diam}_a\,(\{x_n\})=:\limsup_k \sup\{\Vert x_n-x_m\Vert: n,m\geq k\}$.
	
\end{definition}
\bigskip

The \emph{modulus of convexity} of $X$ is the function $\delta:[0,2]\to [0,1]$ defined by
$$
\delta(\varepsilon)=\inf \left\{1-\left\Vert \frac{x-y}{2}\right\Vert: x,y\in B(0,1) \text { with } \Vert x-y\Vert\geq \varepsilon\right\},
$$
and the \emph{characteristic of convexity} of $X$ is the number
$$
\varepsilon_0(X)=\sup\{\varepsilon\geq 0:\delta(\varepsilon)=0\}.
$$
It is known that $\varepsilon_0(X)=0$ if and only if $X$ is uniformly convex, while $\varepsilon_0(X)< 1$ implies that $X$ has uniform normal structure.

\bigskip
\begin{definition}
	Let $X$ be a Banach space and $\phi$ a measure of noncompactness on $X$. The modulus of noncompact convexity associated to $\phi$ is defined in the following way:
$$
\Delta_\phi(\varepsilon)=\inf\{1-d(0,A): A\subset B(0,1) \text { is convex }, \phi(A)\geq \varepsilon\}.
$$
\end{definition}

The characteristic of noncompact convexity of $X$ associated with the measure $\phi$ is defined by
$$
\varepsilon_\phi(X)=\sup\{\varepsilon\geq 0:\Delta_\phi(\varepsilon)=0\}.
$$
The space $X$ is said to be \emph{nearly uniformly convex} if $\varepsilon_\phi(X)=0$. We shall use the Kuratowski measure of  noncompact convexity defined for a nonempty bounded subset $A$ of $X$ as follows:

 \begin{eqnarray*}
\alpha(A) = \inf\{\varepsilon> 0&:& A \text { can be covered by a finite number of sets}\\& & \text { with  diameter smaller than } \varepsilon \}\end{eqnarray*}

\bigskip

\section{Pointwise asymptotic conditions}

There is  a  class of mappings which lies between the class of eventually nonexpansive mappings and the class of mappings which are eventually nonexpansive in the  weak sense  considered in the introduction.  It was introduced in  \cite{KiX} (see also \cite{Ki3}) .

\begin{definition}
	A mapping $T:C\to C$ is said to be {\rm pointwise  eventually nonexpansive} if for every $x\in C$ there exists $N(x)\in \mathbb{N}$ such that if $n\geq N(x)$
	$$\Vert T^n x-T^n y\Vert \leq \Vert x-y\Vert\,\,\,\, \textrm{ for all } y\in C.$$
\end{definition}

\begin{remark}\label{remark1}
	Notice that if $N=N(x)$, the mapping $T^N$ is continuous at $x$. We could guess that it should be possible to find an iterated which is nonexpansive, and so simultaneously continuous for all $x\in C$.  However, the following example shows a  pointwise eventually nonexpansive mappings  which does not have any continuous iterated.
\end{remark}

\begin{example} (A pointwise eventually nonexpansive mapping which does not have a continuous iterated).
	
	 Let $X=\ell_2$, $u_k=e_k/2^k$ and $C$ the convex compact set
	 $\overline{\textrm{co }} \{ u_k: k\in \mathbb{N}\}.$
	  Note that every $x=(x(k))\in C$ has the representation $\displaystyle x=\sum_{i=1}^\infty  \lambda_ku_k$ where $0\leq \lambda_k\leq 1$, $\displaystyle \sum_{k=1}^\infty \lambda_k \leq 1$ and $x(k)=\lambda_k/2^k$.
	Denote $P=\{ n\in N: n=2^j $ for some $j\in \mathbb{N}\}$ and $Q=\mathbb{N}\setminus P$. We write $\mu_k=0$ if $\lambda_k<1$ and $\mu_k= 1/2$ if $\lambda_k=1$. Define
	$$S\left( \sum_{k=1}^\infty \lambda_k u_k\right)= \sum_{k\in Q} \lambda_k u_{k+1} + \sum_{k\in P} \mu_k u_{k+1}.$$

	We will prove that for every $x\in C$, there exists $n(x)\in \mathbb{N}$  such that for every $n\geq n(x)$ there is a number $\alpha_n(x)\in [0,\infty)$ satisfying $\lim_n \alpha_n(x)=0$ and
	$$\Vert T^n x-T^ny\Vert \leq \alpha_n(x) \Vert x-y\Vert.$$
	In particular, $T$ is pointwise eventually nonexpansive.
	
	We will distinguish two cases:

	\begin{itemize}
	\item [(1)] Assume $x\notin \{u_k:k\in \mathbb{N\}}$.  We split again in two cases.

\medskip
	
	(1a) Assume $y\notin \{u_k:k\in \mathbb{N\}}$. For any $k\in \mathbb{N}$, we have that either $(T^nx)(k)=(T^ny)(k)=0$ or $(T^nx)(k)=(1/2^n)x(k)$ and $(T^ny)(k)=(1/2^n)y(k)$. Thus, $\Vert T^n x-T^n y\Vert \leq (1/2^n)\Vert x-y\Vert$.
	
\medskip

	(1b) Assume $y=u_k$ for some $k\in \mathbb{N}$. Denote $a=\min\{ 1-2^i x(i): i\in \mathbb{N}\}$. We have $\vert y(k)-x(k)\vert \geq a/(2^k)$. For $j\not= k$  we have $\vert (T^n x)(j)-(T^ny)(j) \vert =\vert (T^n x)(j)\vert \leq (1/2^n)\vert x(j)\vert \leq (1/2^n) \vert x(j)-y(j)\vert$. For $j=k$ we have
	\begin{eqnarray*} \vert (T^n x)(k)-(T^ny)(k) \vert &\leq & \vert (T^n x)(k)\vert +\vert (T^ny)(k) \vert \leq\frac{2}{2^{n+k}}\\&=&\frac{a}{a2^{n+k-1}}\leq  \frac{1}{2^{n-1}a} \vert x(k)-y(k)\vert. \end{eqnarray*}
	Thus,
	$$\Vert T^n x-T^n y\Vert \leq \frac{1}{ 2^{n-1}a}\Vert x-y\Vert.$$
	
	\item [(2)] Assume $x=u_k$. Let $q$ be an  integer in $P$ greater than $k$. For $n>q$ we have $T^n(u_k)=0$. Thus, $(T^nx)(k)=(T^ny)(k)=0$. For $j\not= k$ we have
	$$\vert (T^nx)(j)-(T^ny)(j)\vert =\vert (T^ny)(j)\vert \leq \frac{1}{2^n}\vert y(j)\vert = \frac{1}{2^n}\vert x(j)- y(j)\vert.$$
	Thus, $\Vert Tx-Ty\Vert \leq (1/2^n)\Vert x-y\Vert$.
\end{itemize}
	\medskip
	
	To finish, we will show that for each $n\in \mathbb{N}$, $T^n$ is discontinuous at infinitely many  points of $C$. Indeed, for a given  $n\in \mathbb{N}$, choose $k=2^j\in P$ such that $n<2^j$. We have
	$T^n (u_k)=2^{-(k+1)} u_{k+n}$ and $T(\lambda u_k)=0$ for every $\lambda \in (0,1)$.
\end{example}

Related to the notion of asymptotically nonexpansive mapping \cite{GK1}, the following concept has been considered:

\begin{definition}[\cite{Ki3}]
	A mapping $T:C\to C$ is said to be {\rm pointwise  asymptotically nonexpansive} if  for each $x\in C$ there exist $N(x)\in \mathbb{N}$ and a real sequence $\alpha_n(x)$ such that if $n\geq N(x)$
	$$\Vert T^n x-T^n y\Vert \leq \alpha_n(x)\Vert x-y\Vert\,\,\,\, \textrm{ for all } y\in C, $$
	where $\dis \lim_n \alpha_n(x)=1$.
\end{definition}

A more general class of mappings has been considered  by Kirk \cite{Ki1}:

\begin{definition}
	A mapping $T:C\to C$ is said to be an {\rm  asymptotically nonexpansive type } mapping if for each $x\in C$,
	$$
	\limsup_{n\to \infty}\{\sup\{\Vert T^nx-T^ny\Vert-\Vert x-y\Vert:y\in C\}\}\leq 0.
	$$
\end{definition}

 \begin{remark}It is easy to check that if $C$ is bounded then a pointwise asymptotically nonexpansive mapping $T$ is an asymptotically nonexpansive type mapping.
 	\end{remark}

Although it is not yet known whether the FPP is equivalent to the FPP  for pointwise eventually nonexpansive mappings,  some classical existence results for fixed points of nonexpansive mappings have been extended to some classes of  asymptotically nonexpansive type mappings.  Most of them assume, in addition,  the continuity of an iterated.

\medskip

In 1974 by Kirk \cite{Ki1} proved the following fixed point theorem in a  class of Banach spaces satisfying a property weaker than uniform convexity.

 \begin{theorem}[\cite{Ki1}]
 	Let $X$ be a Banach space for which $\varepsilon_0(X)< 1$, $C$  a nonempty bounded closed and convex subset of $X$, and $T:C\to C$ an asymptotically nonexpansive type mapping.  Suppose that there exists an integer $N\geq 1$ such that $T^N$ is continuous. Then $T$ has a fixed point.
 \end{theorem}

In 2000, Kim and Xu \cite{KX} demonstrated that the uniform normal structure of the space $X$ implies the existence of fixed points for asymptotically nonexpansive mappings. Shortly after, this result was  extended to the class of  asymptotically nonexpansive type mappings \cite{GS}.

\begin{theorem}[\cite{GS}]
		Let $X$ be a Banach space with uniform normal structure, $C$  a nonempty bounded closed and convex subset of $X$, and $T:C\to C$ an asymptotically nonexpansive type mapping such that $T$ is continuous. Then $T$ has a fixed point.
\end{theorem}

The same result  had already been proved by Xu \cite{Xu}  in a nearly uniformly convex Banach space. At this point, it should be remembered that nearly uniform convexity implies normal structure but does not imply uniform normal structure (see pages 78-79 in \cite{GK2}).
\begin{theorem}[\cite{Xu}]
	Let $X$ be nearly uniformly convex Banach space, $C$  a nonempty bounded closed and convex subset of $X$, and $T:C\to C$ an asymptotically nonexpansive type mapping.  Suppose that there exists an integer $N\geq 1$ such that $T^N$ is continuous. Then $T$ has a fixed point.
\end{theorem}

Bearing in mind Example \ref{e2}, the continuity assumption cannot be dropped in the preceding theorems.
Notice that the mapping in Example \ref{e2} is an asymptotically nonexpansive type mapping but it is not a pointwise eventually asymptotically nonexpansive mapping. Motivated by this fact, some natural questions have been raised (see \cite{Ki3, KiX}): Does a pointwise eventually nonexpansive mapping (or more generally, a pointwise asymptotically nonexpansive mapping)
defined on a weakly compact convex subset $C$ of a Banach space $X$ have a fixed point if $X$ satisfies some of the following conditions
\begin{itemize}
	\item (1) $X$ is a  Banach space with the FPP
	
	\item (2) $X$ has uniform normal structure
	
	\item (3) $X$ is nearly uniformly convex
	
	\item (4) The asymptotic center relative to $C$ of each sequence in $C$ is compact?
\end{itemize}

A positive answer to Question (2) is given   in \cite[Theorem 3.4]{RR} for pointwise asymptotically nonexpansive mappings. Nevertheless, we will see that this theorem is actually a consequence of the proof carried out   to prove Theorem 2.1 in \cite{GS}.

 Our main results in  the last section of this paper  will give a positive answer to Questions (3) and (4). Note that both questions are independent. Indeed, in \cite{KP}, an example is shown of a Banach space which is nearly uniformly convex but asymptotic centers of bounded sequences are not, in general, compact. On the other hand, the following example shows a Banach spaces which fails to be nearly uniformly convex, but it is uniformly rotund in every direction (see \cite {Zi}) which implies that any asymptotic center of a bounded sequence is a singleton.

 \begin{example} Assume that $X$ is the $\ell_2$ product of the space $\ell_k$, $k\geq 2$, i.e. $X=\{(x(k)): x(k)\in \ell_k\textrm{ such that } \sum_{k=2}^\infty \Vert x(k)\Vert_k^2<\infty\}$ where $\Vert x(k)\Vert_k $ is the norm of the vector $x(k)$ in $\ell_k$.  This space can be equipped with the norm
 $$\Vert (x(k))\Vert=\left( \sum_{k=2}^\infty \Vert x(k)\Vert_k^2\right)^{1/2}.$$
 Since $X$ contains isometrically $\ell_k$ for every $k\geq 2$ and $WCS(\ell_k)=2^{1/k}$ we have $WCS(X)=1$ which implies that $X$ fails to be nearly uniformly convex (see, for instance,  \cite[Chapter 6]{ADL}).  However, $X$ is uniformly rotund in every direction. Indeed, let $\{x_n\}$ be a bounded sequence  and $z$ a normalized vector in $X$. Assume that
 $$\lim_n 2\left( \Vert x_n+z\Vert ^2 +\Vert x_n\Vert^2\right)-\Vert 2x_n+z\Vert^2 =0.$$
 Choose $k\geq 2$ such that $z(k)$ is a non-null vector of $\ell_k$. Denote $\tilde z = (\tilde z(j))$ where $\tilde z (j)=z(j)$ if $j\not= k$ and $\tilde z(k)=0$. Define analogously $\tilde x_n$. We have
 $$ \Vert x_n+z\Vert^2=\Vert \tilde x_n+\tilde z\Vert^2 + \Vert x_n(k)+z(k)\Vert_k^2;$$
 $$ \Vert 2x_n+z\Vert^2=\Vert 2\tilde x_n+\tilde z\Vert^2 + \Vert 2x_n(k)+z(k)\Vert_k^2;$$
 $$ \Vert x_n\Vert^2=\Vert \tilde x_n\Vert^2 + \Vert x_n(k)\Vert_k^2;$$
 $$2\left( \Vert \tilde x_n+\tilde z\Vert ^2 +\Vert \tilde x_n\Vert^2\right)-\Vert 2\tilde x_n+ \tilde z\Vert^2 \geq 0.$$
 Thus,  $$\lim_n 2\left( \Vert x_n(k)+z(k)\Vert ^2_k +\Vert x_n(k)\Vert^2_k\right)-\Vert 2x_n(k)+z(k)\Vert_k^2 =0$$
 which is a contradicition because $\ell_k$ is uniformly convex (see \cite[Proposition 1]{Zi}).

 \end{example}

\section{Previous lemmas}

Let us establish some  lemmas which will be used in our main result. Assume that  $C$ is a nonempty weakly compact convex subset of a Banach space $X$ and $T:C\to C$  an asymptotically nonexpansive type mapping.

For each $x\in C$, let us denote by $\omega(x)$  the cluster point set of the sequence $\{T^nx\}$ for the norm topology.  Similarly, $\omega_w(x)$ will be the cluster point set of the sequence $\{T^nx\}$ for the weak topology.

\begin{remark}
	Note that  the cluster point set $A$ of a sequence $\{x_n\}$ is a closed set in any metric space. Indeed if $\{u_k\}$ is a sequence in $A$ convergent to $u$, we can inductively choose  integers $n_k>n_{k-1}$ such that $d(x_{n_k},u_k)<1/k$. The sequence $\{x_{n_k}\}$ converges to $u$ and $u\in A$.
\end{remark}

 According to Remark 1 in \cite{Xu},  there exists a closed convex nonempty subset $K$ of $C$ such that  $\omega_w(x) \subset K$ for every $x\in K$ and it is minimal under these conditions.
Indeed, denote $ \mathfrak{F}$ the collection formed by all closed convex nonempty subsets $D$ of $C$ which contain $\omega_w(x)$ for every $x\in D$. Let $ \mathfrak{F}$ be ordered by inclusion. Then, $C\in \mathfrak{F}$ and for every chain $\{D_i:I\in I\}$ in  $\mathfrak{F}$ we have that $\cap_{i\in I} D_i \in  \mathfrak{F}$. By Zorn's lemma, we obtain a minimal set $K$ in $ \mathfrak{F}$.

 To prove our main result (see Lemma \ref{main} in this paper) we will use the following lemma from \cite {Xu}.

\begin{lemma}[\cite{Xu}] \label{Xu} Let $C$ be a weakly compact convex subset of a Banach space $X$, $T:C\to C$ an asymptotically nonexpansive type mapping and $K$ a closed convex nonempty subset of $C$ such that the cluster point set $\omega_w(x)$ of the sequence $\{T^nx\}$ is contained in $K$ for every $x\in K$ and the set $K$ is minimal under these conditions. Then there exists $\rho\geq 0$ such that
	$$\limsup_n \Vert T^n x-y\Vert=\rho $$
for every $x,y\in K$. \end{lemma}

\begin{lemma} \label{image} Let $C$ be a weakly compact convex subset of a Banach space $X$, $T:C\to C$ an asymptotically nonexpansive type mapping.  Assume that $H$ is a closed subset of $C$ such that for every $x\in H$, every subsequence of $\{T^n x\}$ has a  further (norm)-convergent subsequence.  Then,  for every $x\in H$, $\omega(x)$ is a compact set. Furthermore, if $x_0\in H$ and $x\in \omega(x_0)$, for every increasing sequence $\{h_k\}$ of positive integers, there exist  a subsequence, again denoted $\{h_k\}$,   and a sequence $\{v_k\}$ in $\omega(x_0)$ such that $T^{h_k} v_k=x$. \end{lemma}

\begin{proof}  Indeed, if $\omega(x)$ were not compact, we could find a $d$-separated sequence $\{x_k\}$ in $\omega(x)$. By induction, we choose $n_k>n_{k-1}$ such that $\Vert x_k-T^{n_k} x\Vert < d/3$. Then, $\{T^{n_k} x\}$ is a $d/3$-separated subsequence of $\{T^n x\}$ contradicting the assumptions. Assume that $\{T^{n_i}x_0\}\to x$. For any fixed $k$, we can take a subsequence $\{n_i(k)\}_i$ of $\{n_i\}$ such that the sequence $\{T^{n_i(k)-h_k} x_0\}_i$ converges, say to $v_k$. By compactness, there is a convergent subsequence of $\{v_k \}$, denoted again   $\{v_k \}$,   say to $v$.
For an arbitrary $\varepsilon >0$  there exists $n_0\in \N$,  such that for all positive integer $n$ greater than or equal to $n_0$, we have
$$\sup\{\Vert T^n{v}-T^n{u}\Vert -\Vert v-u\Vert:u\in C\}< \varepsilon/6.$$

 We can choose $k$, large enough, such that $h_k\geq n_0$ and  $\Vert v-v_k\Vert < \varepsilon/6$. Fixing such a  $k$, we can choose $n_i(k)$, large enough, such that $\Vert T^{n_i(k)} x_0-x\Vert <\varepsilon/6$, $\Vert T^{n_i(k)-h_k} x_0-v_k\Vert< \varepsilon/6 $. Thus, we have
\begin{eqnarray*}
\Vert x-T^{h_k}v_k\Vert &\leq & \Vert  x-T^{n_i(k)}x_0\Vert+\Vert T^{n_i(k)}x_0-T^{h_k}v\Vert +\Vert T^{h_k}v -T^{h_k}v_k\Vert\\ &\leq & \varepsilon/6+\Vert T^{n_i(k)-h_k} x_0-v\Vert+\varepsilon/6 +\Vert v-v_k
\Vert +\varepsilon/6 \\ &\leq &
 \Vert T^{n_i(k)-h_k} x_0-v_k\Vert +2\Vert v -v_k\Vert +3\varepsilon/6 \\ &\leq &  \varepsilon.\end{eqnarray*}
Since $\varepsilon$ is arbitrary we have $T^{h_k} v_k=x$.
\end{proof}

\begin{lemma} \label{invariant} Let $C$ be a weakly compact convex subset of a Banach space $X$, $T:C\to C$ an asymptotically nonexpansive type mapping. Let $x_0$ be an arbitrary point in $C$. Then, for every $z\in \omega(x_0)$  we have $\omega (z)\subset  \omega(x_0)$. \end{lemma}
\begin{proof} Assume $\lim_k T^{n_k}x_0=z$ and  $\lim_i T^{n_i}z=y$. %For $n_i>h(z)$ we have
%\begin{eqnarray*}
%\Vert T^{n_k+n_i} x_0-y\Vert &\leq& \Vert T^{n_k+n_i} x_0-T^{n_i}z\Vert +\Vert T^{n_i} z-y\Vert \\ &\leq & \alpha_{n_i}(z)\Vert T^{n_k} x_0-z\Vert +\Vert T^{n_i} z-y\Vert.
%\end{eqnarray*}
For an arbitrary $\varepsilon >0$,  there exists $n_0\in \N$  such that for all positive integer $n$ greater than or equal to $n_0$, we have
$$\sup\{\Vert T^n{z}-T^n{u}\Vert -\Vert z-u\Vert:u\in C\}< \varepsilon/3.$$

Choose $k_0$ such that $\Vert T^{n_k} x_0 -z\Vert <\varepsilon/3$ for each $k\geq k_0$ and for every $k\geq k_0$ choose $n_i(k) \geq n_0$ such that $n_k+n_i(k)> n_{k-1}+n_i(k-1)$ and  $\Vert T^{n_i(k)}z-y\Vert <\varepsilon/3$.

We have

\begin{eqnarray*}
	\Vert T^{n_k+n_i(k)} x_0-y\Vert &\leq & \Vert T^{n_k+n_i(k)} x_0-T^{n_i(k)}z\Vert+\Vert T^{n_i(k)}z-y\Vert\\ &\leq & \Vert T^{n_k} x_0-z\Vert+\varepsilon/3+\varepsilon/3<\varepsilon,
\end{eqnarray*}
	for every $k\geq k_0$  which implies that $\lim_k  T^{n_k+n_i(k)} x_0=y$ and $y\in \omega (x_0)$.
\end{proof}

The following lemma will play a key role in the proof of our main theorems.

\begin{lemma} \label{main} Let $C$ be a weakly compact convex subset of a Banach space $X$, $T:C\to C$ an asymptotically nonexpansive type mapping. Assume that there exists a closed convex nonempty subset $H$ of $C$ which satisfies

(i) For each $x\in H$, $\omega_w(x) \subset H$.

(ii) For each $x\in H$, every subsequence of  $\{T^n x:n\in \N\}$  has a further convergent subsequence.

Then there exists $z\in H$ such that $\{T^nz\}$ is norm convergent to $z$.
 \end{lemma}
\begin{proof}

As stated above, we can find a closed convex nonempty subset $K$ of $H$ which satisfies $\omega_w(x) \subset K$ for every $x\in K$ and it is minimal under these conditions. Fix $x_0\in K$ and denote $S=\omega(x_0)$ which, by Lemma \ref{image},  is a nonempty compact set. By Lemma \ref{invariant} we have $\omega(z)\subset S$ for every $z\in S$. By applying Zorn's lemma, we can find a closed nonempty subset $S_0$ of $S$ which satisfies $\omega(z)\subset S_0$ for each $z\in S_0$ and it is minimal under these conditions. Note that  Lemma \ref{invariant} and  the minimality of $S_0$ imply that $S_0=\omega(x_1)$ for some (any) $x_1\in S_0$. We will prove that diam $S_0=0$. Indeed, since any compact set has normal structure, if diam $S_0 >0$, the Chebyshev radius of $S_0$ is smaller that the diameter of $S_0$. Thus, there exists a point $z$ in $\overline{\textrm{co}} (S_0)\subset K$  such that $r=\sup\{\Vert z-x\Vert : x\in S_0\} <$ diam $(S_0)$. Let
$$D=\{x\in K: \sup_{y\in S_0} \Vert x-y\Vert \leq r\}.$$
It is clear that $D$ is a nonempty closed convex proper subset of $K$. We will prove that $\omega_w(x) \subset D$ for every $x\in D$  contradicting the minimality of $K$. Indeed, assume that $y=w-\lim_i T^{n_i}x$ and $z\in S_0$. By Lemma \ref{image}, there exists an subsequence of $\{n_i\}$, denoted again $\{n_i\}$,  and a sequence $\{v_i\}$ in $\omega(x_1)=S_0$ such that $T^{n_i} v_i=z$. Hence,
\begin{eqnarray*}
\Vert z-y\Vert &\leq & \liminf_i \Vert z-T^{n_i} x\Vert  =  \liminf_i \Vert T^{n_i} v_i -T^{n_i} x\Vert \\ &\leq&   \limsup_i \Vert T^{n_i} v_i -T^{n_i} x\Vert.
\end{eqnarray*}
Since $T$ is an asymptotically nonexpansive type mapping and $\Vert v_i-x\Vert\leq r$ for every $i$ we have
$$\limsup_i \Vert T^{n_i} v_i -T^{n_i} x\Vert-r\leq \limsup_i \{\Vert T^{n_i} v_i -T^{n_i} x\Vert-\Vert v_i-x\Vert\}\leq 0,$$
which implies  that
$\Vert z-y\Vert\leq r.$
Thus,  $y\in D$ and  diam $S_0=0$. Write $S_0=\{z\}$. Since $\omega(z)\subset S_0$ we have that any convergent subsequence of $\{T^n z\}$ must converge to $z$. From condition (ii) the whole sequence $\{T^n z\}$ converges to $z$.

\end{proof}

\begin{remark}
It should be mentioned that this lemma is a significant  generalization of  Lemma 2 in \cite{Xu}. Indeed, firstly,  the continuity of an iterated $T^N$ for some $N\in \N$ is removed and, secondly, the same conclusion as  in Lemma \ref{main} is obtained  but for the mapping $T$ instead of the iterate $T^N$.

\end{remark}

\section{Fixed points for pointwise eventually asymptotically nonexpansive mappings}

We start this section with a very simple   and general result.

\begin{theorem}\label{cont}
	Let $X$ be an arbitrary  topological space, $M$ a nonempty subset of $X$ and $T$ a mapping from $M$ into $X$. Assume that there exists $x\in M$ such that $\displaystyle \lim_nT^n x= x$ and there exists $N\in \N$ such that $T^N$ is continuous at $x$. Then, $Tx=x$.
\end{theorem}
\begin{proof}
Since $T^n x\to x$ and 	$T^N$ is continuous at $x$, we have $T^Nx=x$. Thus $T^{nN+1}x=Tx$ for all  $n\geq 1$,  which implies that   $Tx=x$.
\end{proof}

As we have before mentioned,  a Banach space such that $\epsilon_0(X)<1$  satisfies that any    asymptotically nonexpansive type mapping $T$ has a fixed point in a weakly compact convex $T$-invariant set   if the    continuity of some iterated of the mapping is assumed.  The same is true in the more general setting of a Banach space with uniform  normal structure \cite{GS}. Looking in detail at their proofs, we find that, in both cases,  it is proved  that there exists $x\in C$ such that  $\{T^nx\}$ converges to $x$. Therefore, we obtain  the following theorems  as an  immediately consequence of Theorem \ref{cont} together with Remark \ref{remark1}.

\begin{theorem}
Let $C$ be a convex bounded subset of a Banach space $X$, $T:C\to C$  an asymptotically nonexpansive type mapping. Assume that $\varepsilon_0(X)<1$. Then,  there exists $x\in C$ such that $T^n x\to x$.  In particular,  $T$ has a fixed point if it is pointwise asymptotically eventually nonexpansive.
\end{theorem}

\begin{theorem}
	 Let $C$ be a convex bounded subset of a Banach space $X$, $T:C\to C$  an asymptotically nonexpansive type mapping. Assume that  $X$ has uniform normal structure. Then,  there exists $x\in C$ such that $T^n x\to x$.  In particular, $T$ has a fixed point if it is pointwise  asymptotically eventually nonexpansive.
\end{theorem}

It should be noted that the above theorem recovers and extends Theorem 3.4 in  \cite{RR}.

\bigskip

Our next fixed point theorems are a direct  consequence of Lemma \ref{main}. When $C$ is a compact set, the assumptions of Lemma \ref{main} are immediately satisfied for $H=C$. Therefore, Theorem \ref{cont} leads to the following result.

\begin{theorem}\label{compact}
	Let $C$ a nonempty, compact, convex subset of a Banach space and $T:C\to C$ an asymptotically nonexpansive type mapping. Assume that for each $x\in C$ the mapping  $T^{N(x)}$ is continuous at $x$ for some $N(x)\in \N$. Then $T$ has a fixed point.
\end{theorem}

 In particular, the above theorem  extends \cite[Theorem 4]{ Ki1}, removing the continuity assumption for an iterated $T^N$. As a consequence we have:

\begin{corollary}
	Let $C$ a nonempty, compact, convex subset of a Banach space and $T:C\to C$ a  pointwise asymptotically nonexpansive mapping.  Then $T$ has a fixed point.
\end{corollary}

 Having in mind Schauder's Fixed Point Theorem, it is interesting to note that we can obtain a fixed point for a discontinuous mapping defined on a convex closed bounded set.

 \medskip

The following theorem solves a  long-standing open question (see \cite[Question 2]{KX} and \cite[Question 2]{Ki3}).

\begin{theorem}
Let $C$ be a weakly compact convex Banach space $X$ and $T:C\to C$  an pointwise eventually  nonexpansive  mapping. Suppose that each sequence in $C$ has  compact asymptotic center relative to $C$. Then $T$ has a fixed point.
\end{theorem}

\begin{proof}
	Choose an arbitrary point $x_0\in C$ and let $A=A(C,\{T^nx_0\})$. It is easy to check that for each $y\in A$,  $T^my\in A$ for all large $m$. Therefore, the weak (in fact, strong) cluster point set of the  $\{T^my\}$ is a subset of $A$. Hence $A$ satisfies condition (i) in Lemma \ref{main}. Since $A$ is compact, it also meets  condition (ii). These facts together with Theorem \ref{cont} lead to the desired result.

\end{proof}

The main result in this section extends Xu's result in \cite{Xu} and gives a positive answer to another  question raised in \cite{KX, Ki3}.

\begin{theorem}\label{NUC} Let $X$ be a nearly uniformly convex Banach space, $C$ a closed convex bounded nonempty subset of $X$, $T:C\to C$ an asymptotically nonexpansive type mapping and  $K$ a closed convex nonempty subset of $C$ such that the cluster point set $\omega_w(x)$ of the sequence $\{T^nx\}$ is contained in $K$ for every $x\in K$ and the set $K$ is minimal under these conditions. Then, $K$ is a singleton $\{x_0\}$ and  $\{T^nx_0\}$ is norm convergent to $x_0$. In particular, every pointwise  asymptotically  nonexpansive mapping $T:C\to C$ has a fixed point.\end{theorem}

\begin{proof} 	From  Lemma \ref{Xu} it suffices to prove that $\rho=0$.
	Assume that $\rho>0$. By construction,  $K$ satisfies (i)  in Lemma \ref{main}.  We will prove that it also satisfies (ii). Otherwise,  assume that there are a point $x\in K$ and  a subsequence  of $\{T^n x\}$ which does no admit a further convergent subsequence. In this case, there exist a number $d >0$ and a $d$-separated subsequence $\{T^{n_i}x\}$ of $\{T^n x\}$. For an arbitrary $\varepsilon >0$  there exists $n_0\in \N$,  such that for all positive integer $n$ greater than or equal to $n_0$, we have
	$$\Vert T^nx-x\Vert <\rho+\varepsilon/2$$
	and $$\sup\{\Vert T^n{x}-T^n{u}\Vert -\Vert x-u\Vert:u\in C\}< \varepsilon/2.$$
	 Fix such an $n\in \mathbb{N}$.  Since the sequence $\{T^{n_i}x-T^n x\}_i$ is $d$-separated we have $\alpha(\{T^{n_i}x-T^n x\}_i)\geq d$. Taking a subsequence, we can assume that $\{T^{n_i}x\}$ converges weakly, say to $x_\infty\in K$. On the other hand,  there exists  $i_0$  large enough such  that $n_i\geq n+n_0$ for $i\geq i_0$. Thus, for  $i\geq i_0$, we have
	 $$
	 \Vert T^nx-T^{n_i}x\Vert=\Vert T^nx-T^nT^{n_i-n}x\Vert \leq  \varepsilon/2+\Vert x-T^{n_i-n}x\Vert\leq r+\varepsilon.
	 $$	
 Hence, by the lower weakly semicontinuity of the norm,
$$\frac{\Vert T^nx-x_\infty\Vert }{ r+\varepsilon} \leq 1-\Delta_{\alpha} \left( \frac {d}{r+\varepsilon}\right )$$
which implies  $r=\limsup_n \Vert T^n x-x_\infty\Vert \leq (r+\varepsilon)(1-\Delta_{\alpha} (d/(r+\varepsilon)))$. Letting $\varepsilon \to 0$ we obtain the contradiction $r\leq r(1-\Delta_{\alpha} ({d/r}^-)) <r$. Thus, $K$ is a singleton $\{x_0\}$ and satisfies (i) and (ii). Hence  $\{T^nx_0\}$ is norm convergent to $x_0$. By Theorem  \ref{cont} and Lemma \ref{Xu},  every pointwise asymptotically nonexpansive mapping $T:C\to C$ has a fixed point.
\end{proof}
\bigskip
\begin{remark} Butsan et al. \cite[Corollary 3.18]{BT} claim that they have proved Theorem \ref{NUC} solving the mentioned question.  However, the proof is strongly based on certain connection between the asymptotic center of a sequence and the modulus of noncompact convexity which are stated in   \cite{GK2}. Actually, they assert that in a nearly uniformly convex space $X$,  there exists a $\lambda\in [0,1)$ such that for each closed bounded convex subset $C$ of $X$ and for each sequence $\{x_n\}\subset C$ the following inequality  holds
	$$
	r(C,\{y_n\})\leq \lambda r(C,\{x_n\}),
	$$
for each sequence $\{y_n\}$ in $A(C,\{x_n\})$.

 However, it should  be emphasized that this inequality is not true unless the regularity of the sequence $\{x_n\}$  is assumed (see \cite[Remark 3.5]{DL}), as the following example shows. Since the proof in \cite{BT} does not allow to take a subsequence, Corollary 3.15 does not follow from Theorem 3.14 in \cite{BT} and the question   still  remained open. For completeness, we  include the counterexample of \cite{DL}.

\begin{example}
	Consider the product space $X=Y\bigotimes\ell_2$ where $Y=(\R^2,\Vert \cdot \Vert_{\infty})$ with the norm
	$$
	\Vert (x,y)\Vert=(\Vert x\Vert_{\infty}^2+\Vert y\Vert_2^2)^{\frac{1}{2}},\,x\in Y, y\in \ell_2.
	$$
	As it is proved in \cite{DL}
	$$
	\Delta_\alpha(\varepsilon)=1-\sqrt{1-\frac{\varepsilon^2}{4}},
	$$
	following that $X$ is a nearly uniform convex space.
	If $z_n\in \R^2$ is the sequence defined by $z_{2n-1}=(-1,0)$ and $z_{2n}=(1,0)$ for each $n\in \N$, we consider the sequence $x_n=(z_n,0)\in X$. Denote $B$ the unit ball of $Y$ and let $C=B\times \{0\}$. Clearly $C$ is a weakly compact (and compact) convex subset of $X$ which contains $\{x_n\}$. It is not difficult to see that $r(C,\{x_n\})=1$ and $A(C,\{x_n\})=\{((0,y),0):y\in[-1,1]\}$. Define the sequence $u_n\in \R^2$ by $u_{2n-1}=(0,-1)$ and $u_{2n}=(0,1)$ for each $n\in \N$ and $y_n=(u_n,0)$. Then $y_n\in A(C,\{x_n\})$ and
	$r(C,\{y_n\})=r(C,\{x_n\})=1$.
\end{example}

\end{remark}

\section{ OPEN QUESTIONS}
Although our results solve several long-standing open problems, as far as we know, the following questions are still open:

\begin{questions}
Does a pointwise eventually nonexpansive mapping,  defined from a weakly compact convex subset $C$ into $C$, have  a fixed point if
$X$ satisfies one of the following conditions

\begin{itemize}

\item(a) $\epsilon_\alpha(X)<1$ where $\epsilon_\alpha(X)$ is the characteristic of noncompact convexity for the Kuratowski measure of noncompactness.

\item(b) $X$ has uniform weak normal structure (i.e $ WCS(X)>1$).

\item (c) $X$ has normal structure.

\item (d) $X$ has the FPP for nonexpansive mappings?

\end{itemize}

\end{questions}

\smallskip
\end{document}